\documentclass{paper}
\usepackage[utf8]{inputenc}
\usepackage{float}
\usepackage{amsmath}
\usepackage{amsthm}
\usepackage{amssymb}
\usepackage{graphicx}
\usepackage[numbers]{natbib}
\usepackage[unicode=true,
 bookmarks=false,
 breaklinks=false,pdfborder={0 0 1},backref=section,colorlinks=false]
 {hyperref}

\makeatletter

\newcommand{\noun}[1]{\textsc{#1}}

\theoremstyle{plain}
    \ifx\thechapter\undefined
      \newtheorem{prop}{\protect\propositionname}
    \else
      \newtheorem{prop}{\protect\propositionname}[chapter]
    \fi
\theoremstyle{definition}
    \ifx\thechapter\undefined
      \newtheorem{defn}{\protect\definitionname}
    \else
      \newtheorem{defn}{\protect\definitionname}[chapter]
    \fi

\usepackage{stmaryrd}
\makeatother

\providecommand{\definitionname}{Definition}
\providecommand{\propositionname}{Proposition}

\begin{document}
\title{Optimal separator for an hyperbola\\
Application to localization}
\author{Luc Jaulin}
\institution{Lab-Sticc, ENSTA-Bretagne}

\maketitle
\textbf{Abstract}. This paper proposes a minimal contractor and a
minimal separator for an area delimited by an hyperbola of the plane.
The task is facilitated using actions induced by the hyperoctahedral
group of symmetries. An application related to the localization of
an object using a TDoA (Time Differential of Arrival) technique is
proposed. 

\section{Introduction}

Consider the quadratic function 
\begin{equation}
f(\mathbf{q},\mathbf{x})=q_{0}+q_{1}x_{1}+q_{2}x_{2}+q_{3}x_{1}^{2}+q_{4}x_{1}x_{2}+q_{5}x_{2}^{2}\label{eq:fqdex}
\end{equation}
where $\mathbf{q}=(q_{0},\dots,q_{5})$ is the parameter vector and
$\mathbf{x}=(x_{1},x_{2})$ is the vector of variables. Equivalently,
we can write the function in a matrix form: 
\begin{equation}
f(\mathbf{q},\mathbf{x})=\mathbf{x}^{\text{T}}\cdot\underset{\mathbf{Q}}{\underbrace{\left(\begin{array}{cc}
q_{3} & \frac{1}{2}q_{4}\\
\frac{1}{2}q_{4} & q_{5}
\end{array}\right)}}\cdot\mathbf{x}+(q_{1}\,\,\,\,q_{2})\cdot\mathbf{x}+q_{0}.
\end{equation}

The zeros of $f(\mathbf{q},\mathbf{x})$ is a conic section (a circle
or other ellipse, a parabola, or a hyperbola). The characteristic
polynomial of the matrix $\mathbf{Q}$ is 
\[
\begin{array}{ccc}
P(s) & = & (s-q_{3})(s-q_{5})-\frac{1}{4}q_{4}^{2}\\
 & = & s^{2}-(q_{3}+q_{5})s+q_{3}q_{5}-\frac{1}{4}q_{4}^{2}
\end{array}
\]
Its discriminant is
\[
\begin{array}{ccc}
\Delta & = & (q_{3}+q_{5})^{2}-4q_{3}q_{5}+q_{4}^{2}\\
 & = & q_{3}^{2}+q_{5}^{2}-2q_{3}q_{5}+q_{4}^{2}\\
 & = & (q_{3}-q_{5})^{2}+q_{4}^{2}
\end{array}
\]
which is always positive. Which means that the matrix $\mathbf{Q}$
has two real values (this is not a surprise since $\mathbf{Q}$ is
symmetric). We will assume here that $\mathbf{Q}$ has eigen values
with different signs. It means that

\[
\begin{array}{cc}
 & \left(q_{3}+q_{5}-\sqrt{\Delta}\right)\left(q_{3}+q_{5}+\sqrt{\Delta}\right)<0\\
\Leftrightarrow & \left(q_{3}+q_{5}\right)^{2}-\Delta<0\\
\Leftrightarrow & \left(q_{3}+q_{5}\right)^{2}-((q_{3}-q_{5})^{2}+q_{4}^{2})<0\\
\Leftrightarrow & q_{3}^{2}+q_{5}^{2}+2q_{3}q_{5}-(q_{3}^{2}+q_{5}^{2}-2q_{3}q_{5}+q_{4}^{2})<0\\
\Leftrightarrow & 4q_{3}q_{5}-q_{4}^{2}<0\\
\Leftrightarrow & \text{det}\mathbf{Q}<0
\end{array}
\]

Define the set
\begin{equation}
\mathbb{X}=\left\{ (x_{1},x_{2}|f(\mathbf{q},\mathbf{x})\leq0\right\} .
\end{equation}
 In our case $\mathbb{X}$ has a boundary which is an hyperbola and
will be called an \emph{hyperbolic area}. In this paper, we propose
an interval-based method \citep{Moore79} to generate an optimal separator
\citep{DesrochersEAAI2014} for the set $\mathbb{X}$. The technique
is similar to that proposed in \citep{jaulin2023ellipse} for ellipses.
This separator will be used to generate an inner and an outer approximations
for $\mathbb{X}$. 

As an application, we will consider the problem of the localization
of an object using a TDoA (Time Difference of Arrival) technique.
TDoA is a classical positioning methodology that determines the difference
between the time-of-arrival of signals. TDoA is often used in a real-time
to accurately calculate the location of some tracked entities.

This paper is organized as follows. Section \ref{sec:Symmetries}
introduces the notion of symmetries that will be used in the construction
of the separators. Section \ref{sec:Cardinal-functions} defines the
concept of cardinal function associated with a set. Section \ref{sec:Separator}
builds the separator for the hyperbolic area using cardinal functions
and symmetries. Section \ref{sec:Application} illustrates the use
of the separator to approximate the set of position for an object
from the measure of pseudo-distances. Section \ref{sec:Conclusion}
concludes the paper.

\section{Symmetries\label{sec:Symmetries}}

The methodology presented in this paper is based on symmetries of
the equation of $f(\mathbf{q},\mathbf{x})=0$. This section defines
the main concepts related to symmetries that will be used.

\subsection{Conjugate pair}

Consider an equation of the form
\begin{equation}
f(\mathbf{q},\mathbf{x})=0.
\end{equation}

The pair of transformations $(\sigma,\gamma)$ is \emph{conjugate}
with respect to $f$ if 
\begin{equation}
f(\gamma(\mathbf{q}),\sigma(\mathbf{x}))=0\Leftrightarrow f(\mathbf{q},\mathbf{x})=0.\label{eq:conjugate}
\end{equation}

\subsection{Hyperoctahedral group}

Transformations that will be consider are limited to the \emph{hyperoctahedral
group} $B_{n}$ \citep{Coxeter99} which is the group of symmetries
of the hypercube $[-1,1]^{n}$ of $\mathbb{R}^{n}$. The group $B_{n}$
corresponds to the group of $n\times n$ orthogonal matrices whose
entries are integers. Each line and each column of a matrix should
contain one and only one non zero entry which should be either $1$
or $-1$. Figure \ref{fig:bn_repr1.png} shows different notations
usually considered to represent a symmetry $\sigma$ of $B_{5}$.
We will prefer the Cauchy one line notation \citep{Wussing07} which
is shorter. We should understand the symmetry $\sigma$ of the figure
as the function:
\begin{equation}
\sigma(x_{1},x_{2},x_{3},x_{4},x_{5})=(-x_{2},x_{1},x_{5},-x_{4},x_{3}).
\end{equation}

\begin{figure}[H]
\begin{centering}
\includegraphics[width=7cm]{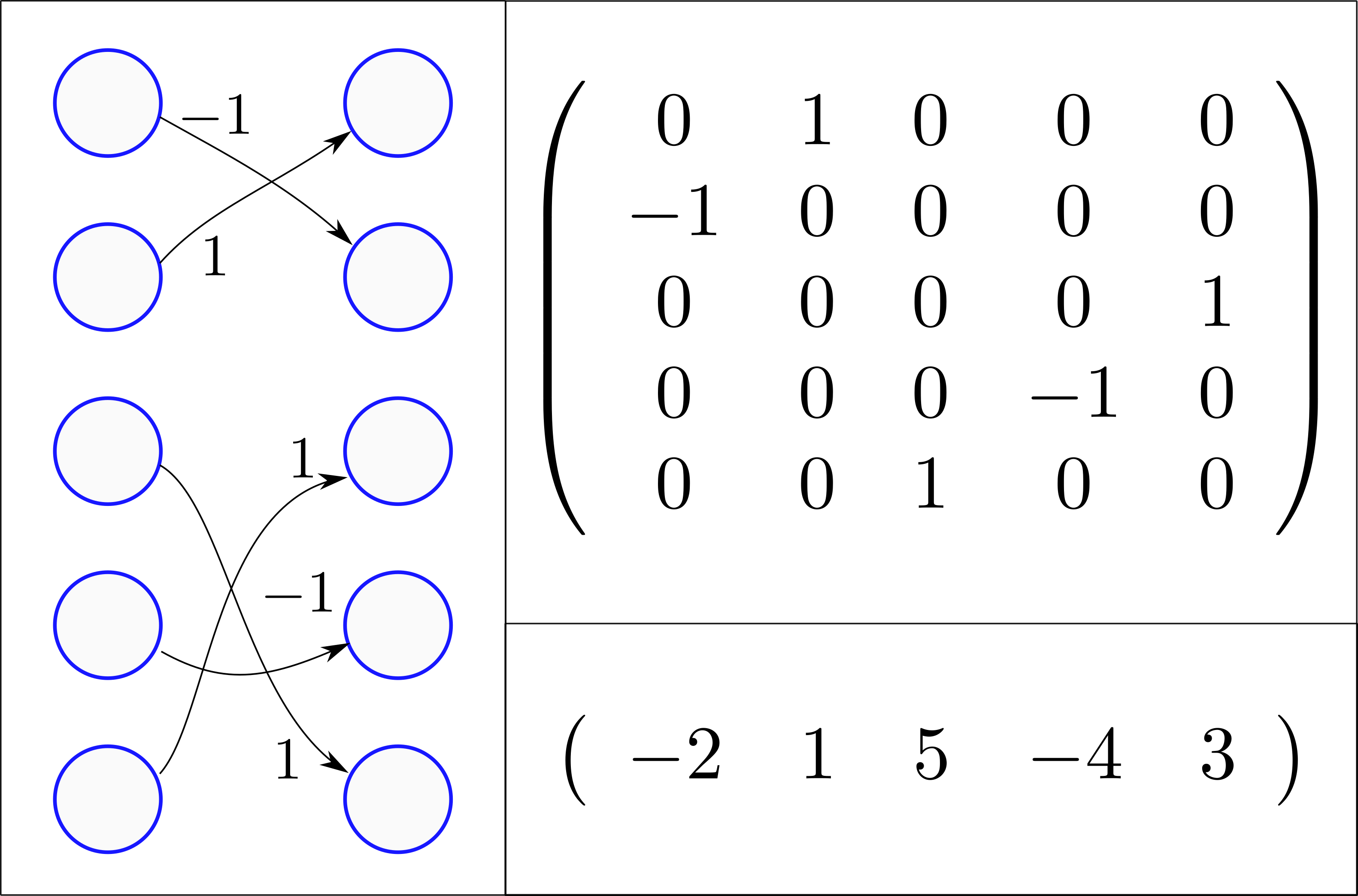}
\par\end{centering}
\caption{Different representations of an element $\sigma$ of $B_{5}$. Left:
graph; Top right: Matrix notation; Bottom right: Cauchy one line notation}
\label{fig:bn_repr1.png}
\end{figure}

In the plane, the group $B_{2}$ has eight elements. If we use the
matrix form, the elements of $B_{2}$ are 
\begin{equation}
\begin{array}{cccc}
\left(\begin{array}{cc}
1 & 0\\
0 & 1
\end{array}\right), & \left(\begin{array}{cc}
-1 & 0\\
0 & 1
\end{array}\right), & \left(\begin{array}{cc}
0 & 1\\
1 & 0
\end{array}\right), & \left(\begin{array}{cc}
-1 & 0\\
0 & -1
\end{array}\right)\\
\left(\begin{array}{cc}
1 & 0\\
0 & -1
\end{array}\right), & \left(\begin{array}{cc}
0 & 1\\
-1 & 0
\end{array}\right), & \left(\begin{array}{cc}
0 & -1\\
1 & 0
\end{array}\right), & \left(\begin{array}{cc}
0 & -1\\
-1 & 0
\end{array}\right)
\end{array}\label{eq:s0:s7}
\end{equation}
Equivalently with the Cauchy notation, these 8 elements of $B_{2}$
are respectively
\begin{equation}
\left\{ (1,2),(-1,2),(2,1),(-1,-2),(1,-2),(-2,1),(2,-1),(-2,-1)\right\} .
\end{equation}

A symmetry of $B_{2}$ in a matrix form, satisfies
\begin{equation}
\sigma=\left(\begin{array}{cc}
\sigma_{11} & \sigma_{12}\\
\sigma_{21} & \sigma_{22}
\end{array}\right)
\end{equation}
with $\sigma_{ij}^{2}\in\{0,1\},\sigma_{i1}^{2}+\sigma_{i2}^{2}=1,\sigma_{1j}^{2}+\sigma_{1j}^{2}=1.$
The Cauchy form is obtained from the matrix form by left multiplying
by the line vector $(1,2):$
\begin{equation}
\sigma=\left(\begin{array}{cc}
1 & 2\end{array}\right)\cdot\left(\begin{array}{cc}
\sigma_{11} & \sigma_{12}\\
\sigma_{21} & \sigma_{22}
\end{array}\right)=(\sigma_{11}+2\sigma_{21},\sigma_{12}+2\sigma_{22}).
\end{equation}

\subsection{Hyperbolic symmetries}

The following theorem gives the symmetries of the hyperbola. 
\begin{prop}
\label{prop:gamma}Take a point $\mathbf{x}=(x_{1},x_{2})$ such 
\begin{equation}
f(\mathbf{q},\mathbf{x})\overset{(\ref{eq:fqdex})}{=}q_{0}+q_{1}x_{1}+q_{2}x_{2}+q_{3}x_{1}^{2}+q_{4}x_{1}x_{2}+q_{5}x_{2}^{2}=0
\end{equation}
and a symmetry 
\begin{equation}
\sigma=(\sigma_{11}+2\sigma_{21},\sigma_{12}+2\sigma_{22})\in B_{2}.
\end{equation}
Define

\begin{equation}
\begin{array}{ccc}
\gamma & = & (q_{0},\sigma_{11}q_{1}+\sigma_{21}q_{2},\sigma_{12}q_{1}+\sigma_{22}q_{2},\sigma_{11}^{2}q_{3}\\
 &  & +\sigma_{21}^{2}q_{5},\sigma_{11}\sigma_{22}+\sigma_{12}\sigma_{21})q_{4},\sigma_{12}^{2}q_{3}+\sigma_{22}^{2}q_{5})
\end{array}
\end{equation}
The pair $(\sigma^{-1},\gamma)$ is conjugate with respect to $f(\mathbf{q},\mathbf{x})$.
\end{prop}
\begin{proof}
Define 
\begin{equation}
\begin{array}{ccc}
x_{1} & = & \sigma_{11}\cdot y_{1}+\sigma_{12}\cdot y_{2}\\
x_{2} & = & \sigma_{21}\cdot y_{1}+\sigma_{22}\cdot y_{2}
\end{array}
\end{equation}
We have
\[
\begin{array}{ccl}
f(\mathbf{q},\mathbf{x}) & = & q_{0}+q_{1}x_{1}+q_{2}x_{2}+q_{3}x_{1}^{2}+q_{4}x_{1}x_{2}+q_{5}x_{2}^{2}\\
 & = & q_{0}+q_{1}(\sigma_{11}y_{1}+\sigma_{12}y_{2})+q_{2}(\sigma_{21}y_{1}+\sigma_{22}y_{2})+q_{3}(\sigma_{11}y_{1}+\sigma_{12}y_{2})^{2}\\
 &  & +q_{4}(\sigma_{11}y_{1}+\sigma_{12}y_{2})(\sigma_{21}y_{1}+\sigma_{22}y_{2})+q_{5}(\sigma_{21}y_{1}+\sigma_{22}y_{2})^{2}\\
 & = & q_{0}+\left(\sigma_{11}q_{1}+\sigma_{21}q_{2}\right)y_{1}+\left(\sigma_{12}q_{1}+\sigma_{22}q_{2}\right)y_{2}+(\sigma_{11}^{2}q_{3}+\sigma_{21}^{2}q_{5})y_{1}^{2}\\
 &  & +(\sigma_{11}\sigma_{22}+\sigma_{12}\sigma_{21})q_{4}y_{1}y_{2}+(\sigma_{12}^{2}q_{3}+\sigma_{22}^{2}q_{5})y_{2}^{2}
\end{array}
\]
Thus
\begin{equation}
\begin{array}{cc}
 & f(\mathbf{q},\mathbf{x})=0\\
\Leftrightarrow & f(\gamma(\mathbf{q}),\mathbf{y})=0\\
\Leftrightarrow & f(\gamma(\mathbf{q}),\sigma^{-1}(\mathbf{x}))=0.
\end{array}
\end{equation}
\end{proof}

\subsection{Choice function}

Considering Proposition \ref{prop:gamma}, we get the choice function
$\psi$ \citep{jaulin:quotient:2023}:
\begin{equation}
\begin{array}{ccl}
\psi_{\sigma}(\mathbf{q}) & = & (q_{0},\alpha_{11}q_{1}+\alpha_{21}q_{2},\alpha_{12}q_{1}+\alpha_{22}q_{2},\\
 &  & \alpha_{11}^{2}q_{3}+\alpha_{21}^{2}q_{5},\alpha_{11}\alpha_{22}+\alpha_{12}\alpha_{21})q_{4},\alpha_{12}^{2}q_{3}+\alpha_{22}^{2}q_{5})
\end{array}\label{eq:psi}
\end{equation}
where $\alpha=\sigma^{-1}$. Given a symmetry $\sigma$, this choice
function allows us to get a symmetry $\gamma$ such that $(\sigma,\gamma)$
is a conjugate pair.

\section{Cardinal functions\label{sec:Cardinal-functions}}

For a given $\mathbf{q}$, the solution set of the equation $f(\mathbf{q},\mathbf{x})=0$
(hyperbola or not) can be decomposed into functions partially defined.
A possible decomposition which works for the hyperbola is based on
cardinal functions to be introduced in this section. 

\subsection{Some definitions}
\begin{defn}
A cardinal vector of $\mathbb{R}^{n}$ is a vector
\begin{equation}
\mathbf{e}=(e_{1},\dots,e_{n})^{\text{T}}
\end{equation}
such that $\|\mathbf{e}\|=1$ and $e_{i}\in\{-1,0,1\}$. 
\end{defn}
For instance $\mathbf{e}_{3}=(0,0,1,0)^{\text{T}}$ and $\mathbf{e}{}_{-2}=(0,-1,0,0)$
are two cardinal vectors of $\mathbb{R}^{4}$. We use the notation
$\mathbf{e}_{i}$ where $i\in I=\{-n,\dots,-1,1,\dots,n\}$ to specify
the cardinal vector. For instance $\mathbf{e}{}_{-2}$ is the vector
parallel to the $2$ axis with a negative direction.
\begin{defn}
Given a closed set $\mathbb{X}$ of $\mathbb{R}^{n}$. A cardinal
function $\varphi_{i}$ with $i\in\{-n,\dots,-1,1,\dots,n\}$ is defined
by 
\begin{equation}
\begin{array}{l}
\varphi_{i}(x_{1},\dots,x_{|i|-1},x_{|i|+1},\dots,x_{n})\\
=\,\,\max\left\{ \mathbf{x}^{\text{T}}\cdot\mathbf{e}_{i}\,|\,\mathbf{x}=(x_{1},\dots,x_{|i|-1},x_{i},x_{|i|+1},\dots,x_{n})\in\mathbb{X}\right\} 
\end{array}
\end{equation}
\end{defn}
Figure \ref{fig:cardfunction1} shows in case of $n=2$, a representation
of the functions $\varphi_{1}(x_{2})$ (red) and $\varphi_{-1}(x_{2})$
(blue). The small squares correspond to cardinal points (East in red
and West in blue). Here, we have two Easts and two Wests.

\begin{figure}[H]
\begin{centering}
\includegraphics[width=8cm]{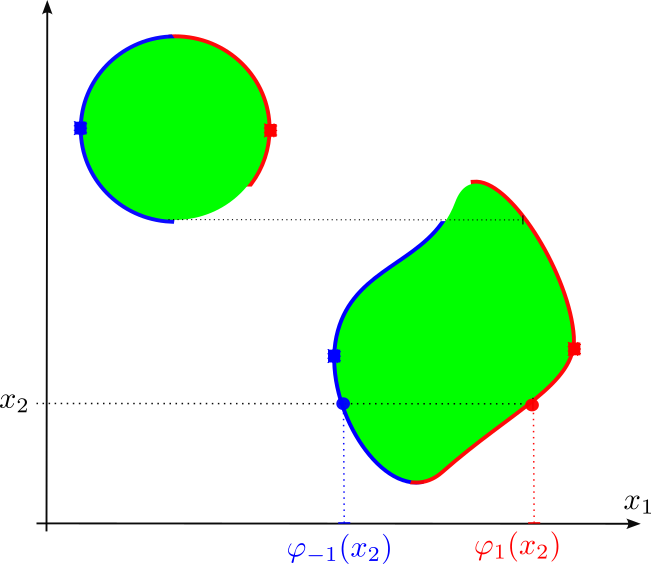}
\par\end{centering}
\caption{Graphs of the functions $\varphi_{1}(x_{2})$ (red) $\varphi_{-1}(x_{2})$
(blue) }
\label{fig:cardfunction1}
\end{figure}

Figure \ref{fig:cardfunction2} is a representation of the functions
$\varphi_{2}(x_{1})$ (black) and $\varphi_{-2}(x_{1})$ (orange).
The small squares correspond to cardinal points (North in black and
South in orange).

\begin{figure}[H]
\begin{centering}
\includegraphics[width=8cm]{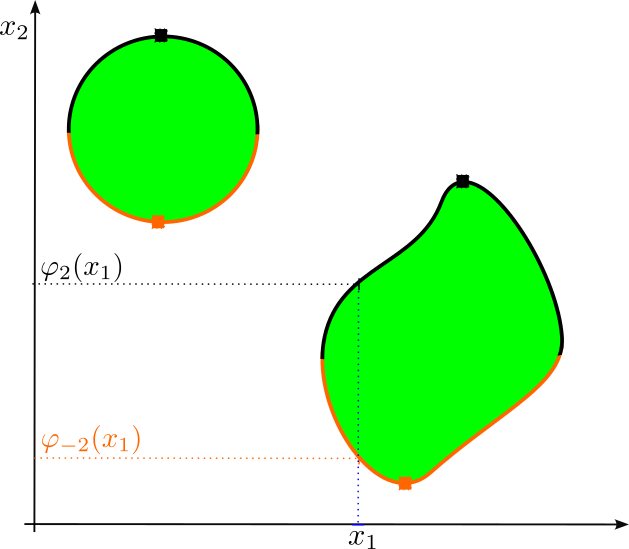}
\par\end{centering}
\caption{Graphs of the functions $\varphi_{2}(x_{1})$ (black) $\varphi_{-1}(x_{2})$
(orange)}
\label{fig:cardfunction2}
\end{figure}

In Figure \ref{fig:cardfunction1}, we observe that graphs of the
function $\varphi_{1}$ and $\varphi_{-1}$ do not cover the boundary
of \noun{$\mathbb{X}$}. This is due to the fact that $\mathbb{X}$
is not row convex. We define the notion of row convexity (similar
to the definition in \citep{Sam-Haroud96})
\begin{defn}
A set $\mathbb{X}\subset\mathbb{R}^{n}$ is said to be \emph{row convex}
if the boundary $\partial\mathbb{X}$ of $\mathbb{X}$ corresponds
to the union of the graphs of its cardinal functions, \emph{i.e.},
\begin{equation}
\partial\mathbb{X}=\cup_{i}\text{graph}(\varphi_{i}).
\end{equation}
\end{defn}

\subsection{Case of the hyperbola}

For the hyperbola defined by
\begin{equation}
f(\mathbf{q},\mathbf{x})=q_{0}+q_{1}x_{1}+q_{2}x_{2}+q_{3}x_{1}^{2}+q_{4}x_{1}x_{2}+q_{5}x_{2}^{2}=0.
\end{equation}
We have four cardinal functions $\varphi_{i},i\in\{-2,-1,1,2\}.$
As it will be shown, the hyperbola is row convex and its cardinal
functions will be sufficient to completely represent the equation.

To find $\varphi_{1}$, we fix $x_{2}$ and we search for the maximal
value for $x_{1}$. The procedure leads to the following proposition.
Other cardinal functions will be obtained by symmetries. 
\begin{prop}
\label{prop:phiq-1}Take a point $x=(x_{1},x_{2})$ such that $f(\mathbf{q},\mathbf{x})=0$.
Given $x_{2}$, the largest $x_{1}$ such that $f(\mathbf{x})=0$
is given by
\begin{equation}
\begin{array}{ccl}
x_{1} & = & \varphi_{1}(\mathbf{q},x_{2})\\
 & = & \frac{-(q_{1}+q_{4}x_{2})+\text{sign}(q_{3})\cdot\sqrt{(q_{1}+q_{4}x_{2})^{2}-4q_{1}(q_{0}+q_{2}x_{2}+q_{5}x_{2}^{2})}}{2q_{3}}
\end{array}\label{eq:x1-1}
\end{equation}
\end{prop}
\begin{proof}
Given $x_{2}$, let us compute the largest possible value for $x_{1}$.
Since
\begin{equation}
f(\mathbf{q},\mathbf{x})=q_{3}x_{1}^{2}+\left(q_{1}+q_{4}x_{2}\right)x_{1}+q_{2}x_{2}+q_{0}+q_{5}x_{2}^{2},
\end{equation}
we get the following discriminant:
\begin{equation}
\Delta_{1}=b_{1}^{2}-4a_{1}c_{1}\label{eq:Delta1}
\end{equation}
where
\begin{equation}
a_{1}=q_{3},\,b_{1}=q_{1}+q_{4}x_{2},\,c_{1}=q_{0}+q_{2}x_{2}+q_{5}x_{2}^{2}
\end{equation}
The largest solution is
\begin{equation}
x_{1}=\frac{-b_{1}+\text{sign}(a_{1})\cdot\sqrt{\Delta_{1}}}{2a_{1}}.
\end{equation}
which corresponds to (\ref{eq:x1-1}).
\end{proof}
\begin{defn}
The cardinal points are the $(x_{1},x_{2})$ which belong to the graph
of at least three cardinal functions $\varphi_{i}$, $i\in\{-2,-1,1,2\}$. 
\end{defn}
For instance a North belongs to the graphs of $\varphi_{1},\varphi_{-1},\varphi_{2}$
and a East belongs to the graphs of $\varphi_{2},\varphi_{-2},\varphi_{1}$.
For our hyperbola we easily find that there exist four cardinal points.
Of course, the cardinal points depend on $\mathbf{q}$.
\begin{prop}
\label{prop:rho}Consider the hyperbola $f(\mathbf{q},\mathbf{x})=0$.
Define the interval function
\begin{equation}
\rho(\mathbf{q})=\frac{-2q_{3}q_{2}+q_{1}q_{4}+[-1,1]\cdot\sqrt{(2q_{3}q_{2}-q_{1}q_{4})^{2}-(4q_{3}q_{5}-q_{4}^{2})(4q_{3}q_{0}-q_{1}^{2})}}{4q_{3}q_{5}-q_{4}^{2}}\label{eq:rho:q}
\end{equation}
If we set $[x_{2}]=[x_{2}^{-},x_{2}^{+}]=\rho(\mathbf{q})$, then
the North of the hyperbola is $(\varphi_{1}(x_{2}^{-}),x_{2}^{-})$
and the South is $(\varphi_{1}(x_{2}^{+}),x_{2}^{+})$. 
\end{prop}
Note that if the square root is not defined, then there is no cardinal
points.
\begin{proof}
A value for $x_{2}$ yields a feasible $x_{1}$ if $\Delta_{1}\geq0$
(see (\ref{fig:cardfunction2})), \emph{i.e.},
\[
\begin{array}{clc}
 & b_{1}^{2}-4a_{1}c_{1} & \geq0\\
\Leftrightarrow & -(q_{1}+q_{4}x_{2})^{2}+4q_{3}(q_{0}+q_{2}x_{2}+q_{5}x_{2}^{2}) & \geq0\\
\Leftrightarrow & (4q_{3}q_{5}-q_{4}^{2})x_{2}^{2}+(4q_{3}q_{2}-2q_{1}q_{4})x_{2}+4q_{3}q_{0}-q_{1}^{2} & \leq0
\end{array}
\]
which is quadratic in $x_{2}.$ The discriminant is
\begin{equation}
\Delta_{2}=b_{2}^{2}-4a_{2}c_{2}
\end{equation}
where
\begin{equation}
\begin{array}{ccc}
a_{2} & = & 4q_{3}q_{5}-q_{4}^{2}\\
b_{2} & = & 4q_{3}q_{2}-2q_{1}q_{4}\\
c_{2} & = & 4q_{3}q_{0}-q_{1}^{2}
\end{array}
\end{equation}
The corresponding values for $x_{2}$ is 
\[
x_{2}=\frac{-b_{2}\pm\sqrt{(\Delta_{2}}}{2a_{2}}.
\]
and the North corresponds to the smallest one and the South to the
largest.
\end{proof}
\textbf{Corollary. }Consider the hyperbola $f(\mathbf{q},\mathbf{x})=0$.
Take the symmetry $\sigma=(1,3,2,6,5,4)$ and set $[x_{1}]=\rho(\sigma(\mathbf{q})$),
the East of the hyperbola is $(x_{1}^{-},\varphi_{2}(x_{1}^{-}))$
and the West is $(x_{1}^{+},\varphi_{2}(x_{1}^{+}))$.
\begin{proof}
The symmetry $\sigma$ permutes $x_{1}$ and $x_{2}$. Then, we apply
Proposition \ref{prop:rho}. The East becomes the North and the West
becomes the South.
\end{proof}

\section{Separator for the hyperbola\label{sec:Separator}}

\subsection{Interval extension of the cardinal function }

Let us assume that $\mathbf{q}$ is fixed. The dependency with respect
to the parameter vector $\mathbf{q}$ will omitted for simplicity.
As defined in the book of Moore \citep{Moore79}, the interval extension
function of $\varphi_{1}(x_{2})$ is
\[
[\varphi_{1}]([x_{2}])=\left[\{x_{1}|\exists x_{2}\in[x_{2}],x_{1}=\varphi_{1}(x_{2})\}\right]
\]
which returns the smallest interval which contains the set $\varphi_{1}([x_{2}])$.
The same definition applies for other cardinal directions to get $[\varphi_{-1}]([x_{2}]),[\varphi_{2}]([x_{1}])$
and $[\varphi_{-2}]([x_{1}])$.

Due to the monotonicity of $\varphi_{1}$ between the cardinal points,
we have 
\[
[\varphi_{1}]([x_{2}])=[\varphi_{1}(\{x_{2}^{-},x_{2}^{+},c_{2}(1),c_{2}(2),\dots\})]
\]
where $\mathbf{c}(1),\mathbf{c}(2),\dots$ are the cardinal points
inside the box $[-\infty,\infty]\times[x_{2}]$.

Take for instance $\mathbf{q}=(-1,5,2,-2,30,-2)$, \emph{i.e.},
\[
f(x_{1},x_{2})=-1+5x_{1}+2x_{2}-2x_{1}^{2}+30x_{1}x_{2}-2x_{2}^{2}.
\]
For an interval sampling $[x_{2}]=\frac{1}{5}\cdot[k,k+1]$, $k\in\mathbb{N}$,
the function $[\varphi_{1}]([x_{2}])$ generates the red boxes of
Figure \ref{fig:Phi1}. If we do the same for $[\varphi_{-1}]([x_{2}])$,
we get the blue boxes. The small black square corresponds to the North
and the small orange square corresponds to the South.

\begin{figure}[H]
\begin{centering}
\includegraphics[width=8cm]{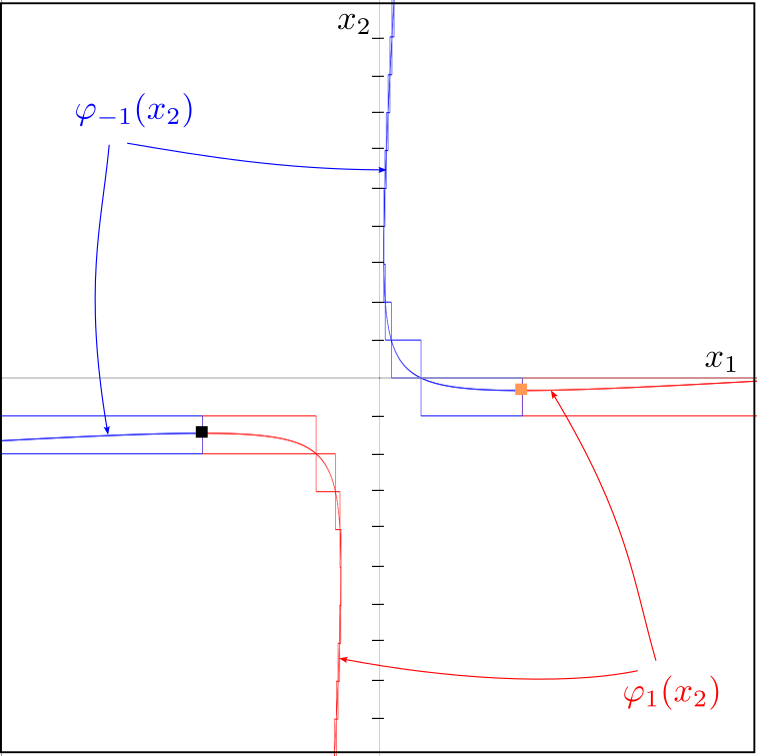}
\par\end{centering}
\caption{Minimal interval extension for $\varphi_{1}([x_{2}])$ (red) and $\varphi_{-1}([x_{2}])$
(blue). The frame box is $[-2,2]^{2}$ }
\label{fig:Phi1}
\end{figure}

For a similar sampling along $x_{1}$, Figure \ref{fig:Phi2} represents
$[\varphi_{2}]([x_{1}])$ (black) and $[\varphi_{-2}]([x_{1}])$ (orange).
The small blue square corresponds to the East and the small red square
corresponds to the West. Note that here, the West is on the left of
the East. This is never the case for an ellipse.

\begin{figure}[H]
\begin{centering}
\includegraphics[width=8cm]{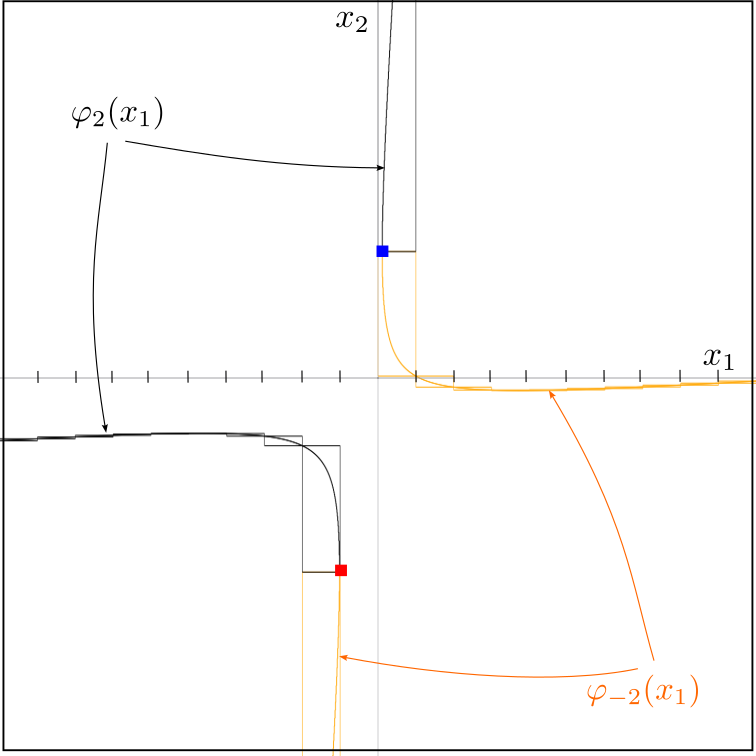}
\par\end{centering}
\caption{Minimal inclusion for $\varphi_{2}([x_{1}])$ (black) and $\varphi_{-2}([x_{1}])$
(orange)}
\label{fig:Phi2}
\end{figure}

\subsection{Seed contractor}

From the interval evaluation, we can build a contractor for the set
$x_{1}=\varphi_{1}(x_{2})$. It is is given by

\begin{equation}
C_{0}:[\mathbf{x}]\rightarrow[x_{1}]\times[\varphi_{1}]([x_{2}]).\label{eq:seed}
\end{equation}
This contractor will be called a \emph{seed contractor} because it
will be used to construct all other contractors using symmetries.
The contractor (\ref{eq:seed}) is not minimal. It is only minimal
with respect to $x_{1}.$ Since this contractor depends on $\mathbf{q}$,
we will write $C_{0}^{\mathbf{q}}.$

We understand that $C_{0}^{\mathbf{q}}$ corresponds to a small portion
of the hyperbola. The main challenge is now to build the the separator
for the whole hyperbola using the single parametric contractor $C_{0}^{\mathbf{q}}$
and symmetries. Of course, we could add some other seed contractors,
but our idea is to factorize the implementation as much as possible
to avoid bugs and make the code adaptable to other types of sets. 

\subsection{Contractor for the hyperbola}

We have a contractor $C_{0}^{\mathbf{q}}$ which is minimal in the
direction of $x_{2}$. Recall that $C_{0}^{\mathbf{q}}([\mathbf{x}])$
contracts the box $[\mathbf{x}]$ with respect to a small portion
of the hyperbola. Using the notion of contractor action \citep{jaulin:IA:hyperoctohedral},
we show how we can extend this contractor $C_{0}^{\mathbf{q}}$ to
other portions. We recall that the action of a symmetry $\sigma$
to the contractor $C$ is defined by
\[
\sigma\bullet C([\mathbf{x}])=\sigma\circ C\circ\sigma^{-1}([\mathbf{x}]).
\]
This means that $\sigma\bullet C$ is a contractor that has been built
from the contractor $C$ as follows:
\begin{itemize}
\item Apply to the box $[\mathbf{x}]$ the symmetry $\sigma^{-1}$
\item Apply the contractor $C$
\item Apply to the resulting box $C\circ\sigma^{-1}([\mathbf{x}])$ the
symmetry $\sigma$.
\end{itemize}
For the hyperbola, we can make a partition of the curve into four
portions : 
\begin{itemize}
\item North-East : $\mathbb{X}^{(1,2)}=\{(x_{1},x_{2}|x_{1}=\varphi_{1}(x_{2})\text{ and \ensuremath{x_{2}=\varphi_{2}(x_{1})}}\}$
\item North-West : $\mathbb{X}^{(1,-2)}=\{(x_{1},x_{2}|x_{1}=\varphi_{1}(x_{2})\text{ and \ensuremath{x_{2}=\varphi_{-2}(x_{1})}}\}$
\item South-East : $\mathbb{X}^{(-1,2)}=\{(x_{1},x_{2}|x_{1}=\varphi_{-1}(x_{2})\text{ and \ensuremath{x_{2}=\varphi_{2}(x_{1})}}\}$
\item South-West : $\mathbb{X}^{(-1,2)}=\{(x_{1},x_{2}|x_{1}=\varphi_{-1}(x_{2})\text{ and \ensuremath{x_{2}=\varphi_{-2}(x_{1})}}\}$
\end{itemize}
If we consider the pair $(\sigma,\gamma)$ conjugate with respect
to the hyperbola, the contractor $\sigma\bullet C_{0}^{\psi_{\sigma}(\mathbf{q})}$
is associated to another portion of the hyperbola. For a given $\sigma$,
the selection of the symmetries $\gamma$ such that $(\sigma,\gamma)$
is conjugate is made using the choice function (\ref{eq:psi}). These
symmetries can be understood geometrically but can also be computed
automatically as shown in \citep{jaulin:IA:hyperoctohedral}. 

To understand the construction, consider the symmetry $\sigma=(2,1)\in B_{2}$.
The contractor associated to $\mathbb{X}^{(1,2)}$:
\[
C_{1}^{\mathbf{q}}([\mathbf{x}])=\left(\sigma\bullet C_{0}^{\psi_{\sigma}(\mathbf{q})}\cap C_{0}^{\mathbf{q}}\right)([\mathbf{x}]).
\]
It is minimal with respect to both directions $x_{1}$ and $x_{2}$
as illustrated by Figure \ref{fig:C01}. Note that the North-East
portion is delimited by the two cardinal points North (black square)
and East (red square). This is consistent with the fact that $\mathbb{X}^{(1,2)}$
corresponds to the North-East portion.

\begin{figure}[H]
\begin{centering}
\includegraphics[width=9cm]{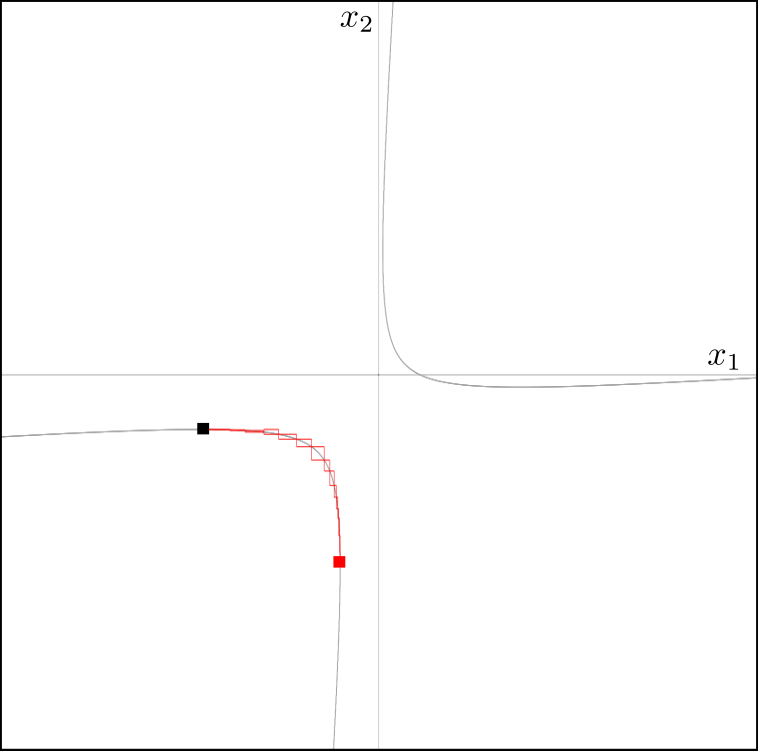}
\par\end{centering}
\caption{Approximation of the North-East portion of the hyperbola using $(2,1)\bullet C_{0}^{\psi_{(2,1)}(\mathbf{q})}\cap C_{0}^{\mathbf{q}}$}
\label{fig:C01}
\end{figure}

The following proposition shows that the contractor for the hyperbola
can be expressed by a simple formula involving symmetries and the
unique seed contractor $C_{0}^{\mathbf{q}}$. Getting such a formula
will ease the implementation of the contractor.
\begin{prop}
\label{prop:bigformula}Consider an hyperbola set $\mathbb{X}$ defined
by $f(\mathbf{q},\mathbf{x})=0$ as given by \ref{eq:fqdex}. A minimal
contractor associated to $\mathbb{X}$ is
\begin{equation}
\bigcup_{^{\sigma\in\{(1,2),(1,-2),(-1,2),(-1,-2)\}}}\sigma\bullet\left((2,1)\bullet C_{0}^{\psi_{(2,1)}\cdot\psi_{\sigma}(\mathbf{q})}\cap C_{0}^{\psi_{\sigma}(\mathbf{q})}\right).\label{eq:ourminimal:ctc}
\end{equation}
where $\psi_{\sigma}(\mathbf{q})$ is the choice function defined
by (\ref{eq:psi}).
\end{prop}
\begin{proof}
The minimal contractor for the North-East portion $\mathbb{X}^{(1,2)}$
is 
\begin{equation}
C_{1}^{\mathbf{q}}=(2,1)\bullet C_{0}^{\psi_{(2,1)}(\mathbf{q})}\cap C_{0}^{\mathbf{q}}.\label{eq:C1q}
\end{equation}
The three other portions can be defined by applying symmetries in
\[
\{(1,-2),(-1,2),(-1,-2)\}.
\]
Define the contractor 
\[
C=\bigcup_{^{\sigma\in\{(1,2),(1,-2),(-1,2),(-1,-2)\}}}\sigma\bullet C_{1}^{\psi_{\sigma}(\mathbf{q})}
\]
Since, the hyperbola is row convex, $C$ is a contractor for $\mathbf{f}(\mathbf{q},\mathbf{x})=0$
(\emph{i.e.} no solution is lost). Moreover, since the union of contractors
is minimal, $C$is minimal. Combining with (\ref{eq:C1q}), we get
that the minimal contractor with respect to the seed contractor $C_{0}^{\mathbf{q}}$
is given by (\ref{eq:ourminimal:ctc}).
\end{proof}
Figure \ref{fig:C0123} illustrates the minimality of the contractor
for the hyperbola.
\begin{figure}[H]
\begin{centering}
\includegraphics[width=9cm]{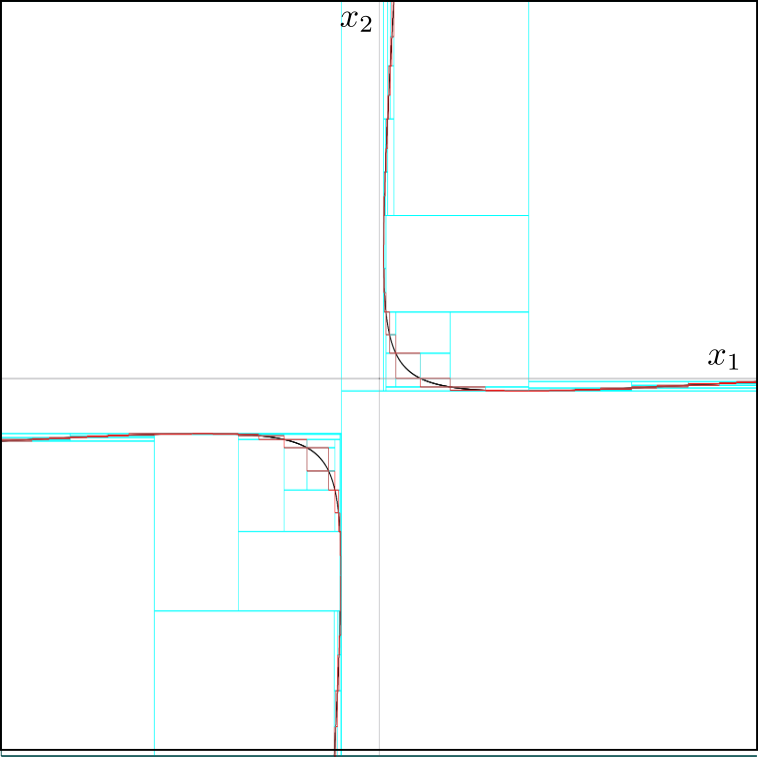}
\par\end{centering}
\caption{Illustration of the minimality of the contractor for the hyperbola}
\label{fig:C0123}
\end{figure}

\subsection{Minimal separator for the hyperbola area}

This section proposes an optimal separator for an hyperbola area defined
by 
\begin{equation}
\mathbb{X}=\{\mathbf{x}|q_{0}+q_{1}x_{1}+q_{2}x_{2}+q_{3}x_{1}^{2}+q_{4}x_{1}x_{2}+q_{5}x_{2}^{2}\leq0\}.
\end{equation}
This separator is then used by a paver to compute boxes that are inside
or outside the solution set.

As shown in \citep{jaulin2023ellipse}, from a contractor on the boundary
of a set $\mathbb{X}$ and a test for $\mathbb{X}$, we can obtain
a separator. As a consequence, we can get an inner and an outer approximations
for $\mathbb{X}$ as illustrated by Figure \ref{fig:siviahyp1} for
$\mathbf{q}=(-1,5,2,-2,30,-2)^{\text{\ensuremath{T}}}.$ The magenta
boxes are proved to be inside $\mathbb{X}$ and the blue boxes are
outside $\mathbb{X}.$ The accuracy is taken as $\varepsilon=0.1$
and corresponds to the size of the small uncertain boxes (yellow).
The cardinal points (North, South, West, East) are represented by
the small squares (black, orange, blue, red). 

\begin{figure}[H]
\begin{centering}
\includegraphics[width=8cm]{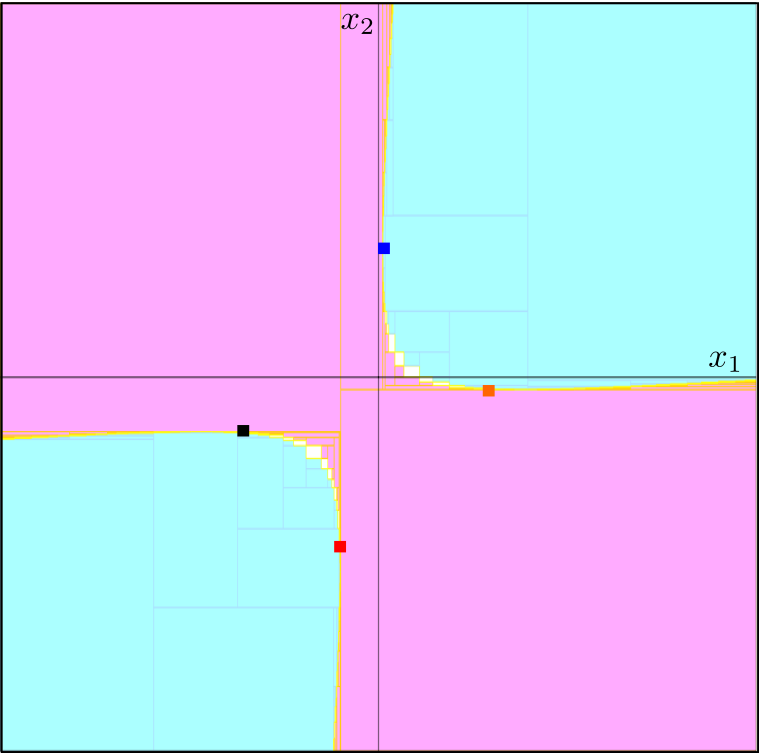}
\par\end{centering}
\caption{Approximation of the hyperbola area obtained by our minimal separator
for the hyperbola set}
\label{fig:siviahyp1}
\end{figure}

Figure \ref{fig:siviahyp3} corresponds to the approximation obtained
with the same accuracy with a classical forward-backward contractor.
The benefice of our method seems small, but we will see later, that
the improvement can become significant when the components of $\mathbf{q}$
are larger.

\begin{figure}[H]
\begin{centering}
\includegraphics[width=8cm]{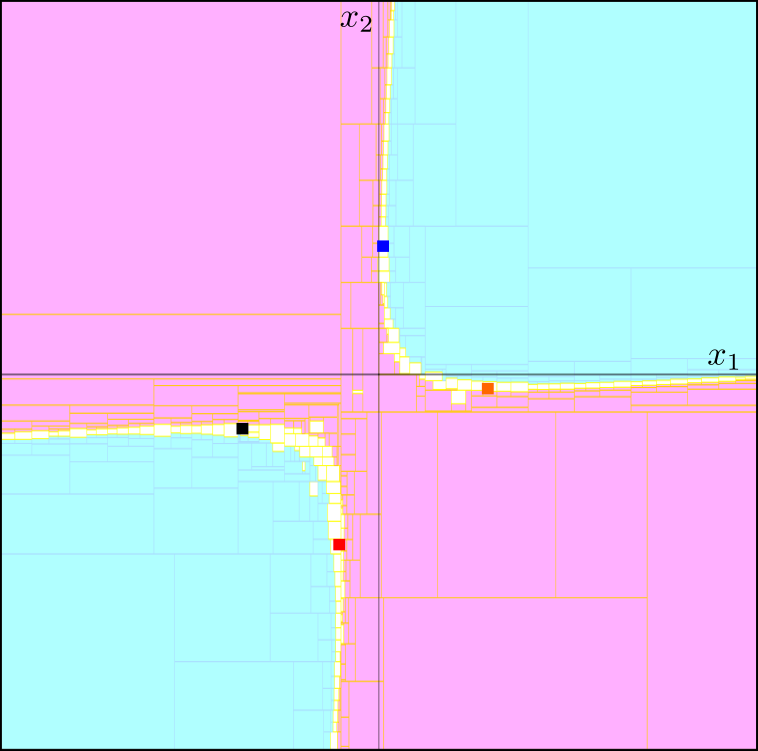}
\par\end{centering}
\caption{Hyperbola area computed using a classical forward-backward contractor}
\label{fig:siviahyp3}
\end{figure}
For $\mathbf{q}=(-1,1,1,3,30,-2)$ we only have two cardinal points
(West and East). The formula provided by Proposition \ref{prop:bigformula}
is still valid and we are able to generate Figure \ref{fig:siviahyp2}.
This shows the ability of the symmetries to consider different situations
easily, elegantly and safely. 

\begin{figure}[H]
\begin{centering}
\includegraphics[width=8cm]{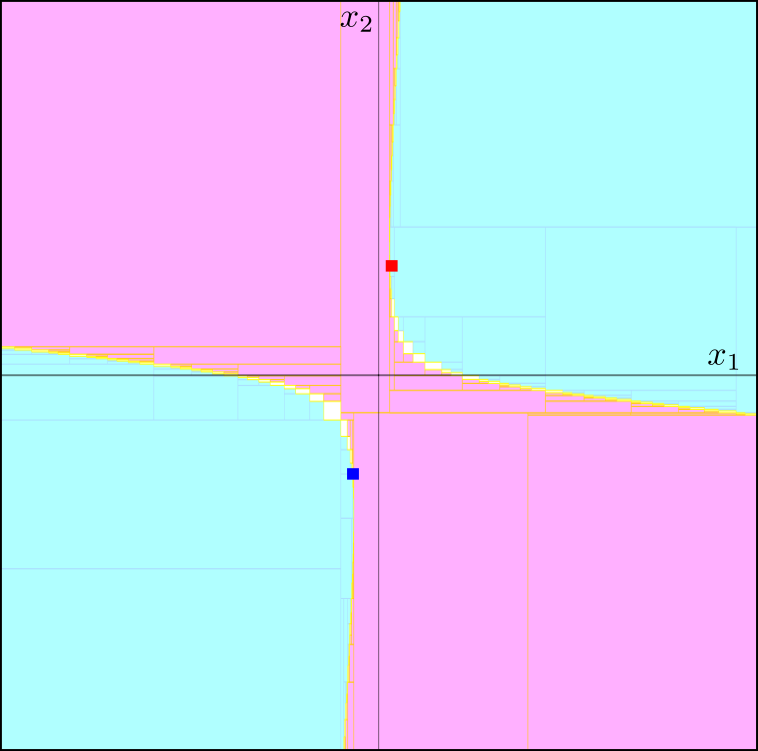}
\par\end{centering}
\caption{Illustration of the application of the separator for an hyperbola
set with two cardinal points only }
\label{fig:siviahyp2}
\end{figure}

\section{Application\label{sec:Application}}

Interval methods have been used for localization of robots for several
decades \citep{JaulinBook01}\citep{robloc}\citep{colle_galerne13}\citep{Drevelle:Bonnifait:09}.
This section proposes to deal with a specific localization problem
where pseudo distances are measured. 

\subsection{Hyperbola from foci}

In this subsection, we show that an equation involving pseudo-distances
corresponds to an hyperbola.
\begin{prop}
\label{prop:hyperbola:foci}Consider two points $\mathbf{a},\mathbf{b}$
of the plane. The set $\mathbb{X}$ of all points such that 
\begin{equation}
\|\mathbf{x}-\mathbf{a}\|-\|\mathbf{x}-\mathbf{b}\|\leq\ell
\end{equation}
 is an hyperbola area with foci $\mathbf{a},\mathbf{b}$. The set
$\mathbb{X}$ is defined by the inequality
\begin{equation}
\mathbf{f}_{\mathbf{a},\mathbf{b},\ell}(\mathbf{x})\leq0
\end{equation}
where
\begin{equation}
\mathbf{f}_{\mathbf{a},\mathbf{b},\ell}(\mathbf{x})=q_{0}+q_{1}x_{1}+q_{2}x_{2}+q_{3}x_{1}^{2}+q_{4}x_{1}x_{2}+q_{5}x_{2}^{2}\label{eq:fabl}
\end{equation}
with
\[
\begin{array}{ccl}
q_{0} & =- & a_{1}^{4}-2a_{1}^{2}a_{2}^{2}+2a_{1}^{2}b_{1}^{2}+2a_{1}^{2}b_{2}^{2}+2a_{1}^{2}\ell^{2}\\
 &  & -a_{2}^{4}+2a_{2}^{2}b_{1}^{2}+2a_{2}^{2}b_{2}^{2}\\
 &  & +2a_{2}^{2}\ell^{2}-b_{1}^{4}-2b_{1}^{2}b_{2}^{2}+2b_{1}^{2}\ell^{2}-b_{2}^{4}+2b_{2}^{2}\ell^{2}-\ell^{4}\\
q_{1} & = & 4a_{1}^{3}-4a_{1}^{2}b_{1}+4a_{1}a_{2}^{2}-4a_{1}b_{1}^{2}-4a_{1}b_{2}^{2}\\
 &  & -4a_{1}\ell^{2}-4a_{2}^{2}b_{1}+4b_{1}^{3}+4b_{1}b_{2}^{2}-4b_{1}\ell^{2}\\
q_{2} & = & 4a_{1}^{2}a_{2}-4a_{1}^{2}b_{2}+4a_{2}^{3}-4a_{2}^{2}b_{2}-4a_{2}b_{1}^{2}\\
 &  & -4a_{2}b_{2}^{2}-4a_{2}\ell^{2}+4b_{1}^{2}b_{2}+4b_{2}^{3}-4b_{2}\ell^{2}\\
q_{3} & = & -4a_{1}^{2}+8a_{1}b_{1}-4b_{1}^{2}+4\ell^{2}\\
q_{4} & = & -8a_{1}a_{2}+8a_{1}b_{2}+8a_{2}b_{1}-8b_{1}b_{2}\\
q_{5} & = & -4a_{2}^{2}+8a_{2}b_{2}-4b_{2}^{2}+4\ell^{2}
\end{array}
\]
 
\end{prop}
\begin{proof}
We have
\begin{equation}
\begin{array}{cl}
 & \|\mathbf{x}-\mathbf{a}\|-\|\mathbf{x}-\mathbf{b}\|=\ell\\
\Rightarrow & \left(\|\mathbf{x}-\mathbf{a}\|-\|\mathbf{x}-\mathbf{b}\|\right)^{2}=\ell^{2}\\
\Leftrightarrow & \|\mathbf{x}-\mathbf{a}\|^{2}+\|\mathbf{x}-\mathbf{b}\|^{2}-2\|\mathbf{x}-\mathbf{a}\|\cdot\|\mathbf{x}-\mathbf{b}\|=\ell^{2}\\
\Leftrightarrow & \|\mathbf{x}-\mathbf{a}\|^{2}+\|\mathbf{x}-\mathbf{b}\|^{2}-\ell^{2}=2\|\mathbf{x}-\mathbf{a}\|\cdot\|\mathbf{x}-\mathbf{b}\|\\
\Leftrightarrow & \left(\|\mathbf{x}-\mathbf{a}\|^{2}+\|\mathbf{x}-\mathbf{b}\|^{2}-\ell^{2}\right)^{2}-4\|\mathbf{x}-\mathbf{a}\|^{2}\cdot\|\mathbf{x}-\mathbf{b}\|^{2}=0
\end{array}
\end{equation}
i.e.,
\begin{equation}
\begin{array}{ccc}
\left((x_{1}-a_{1})^{2}+(x_{2}-a_{2})^{2}+(x_{1}-b_{1})^{2}+(x_{2}-b_{2})^{2}-\ell^{2}\right)^{2}\\
-4\left((x_{1}-a_{1})^{2}+(x_{2}-a_{2})^{2}\right)\left((x_{1}-b_{1})^{2}+(x_{2}-b_{2})^{2}\right) & = & 0
\end{array}
\end{equation}
We can develop the expression to get the coefficients of the proposition. 
\end{proof}

\subsection{Localization}

We consider an example taken from \citep{jaulin2021boundary} related
to localization which can be seen as special case of interval data
fitting problem \citep{Kreinovich:Shary:16}. Consider a robot which
emits a sound at an unknown time $t_{0}$. This sound is received
with a delay by three microphones located points $\mathbf{a}:(13,7),\mathbf{b}:(4,6),\mathbf{c}:(16,10)$
of the plane (see Figure \ref{fig:ellipseloc0}). Taking into account
the time of flight of the sound we want to estimate the position of
the object. 

\begin{figure}[H]
\begin{centering}
\includegraphics[width=7cm]{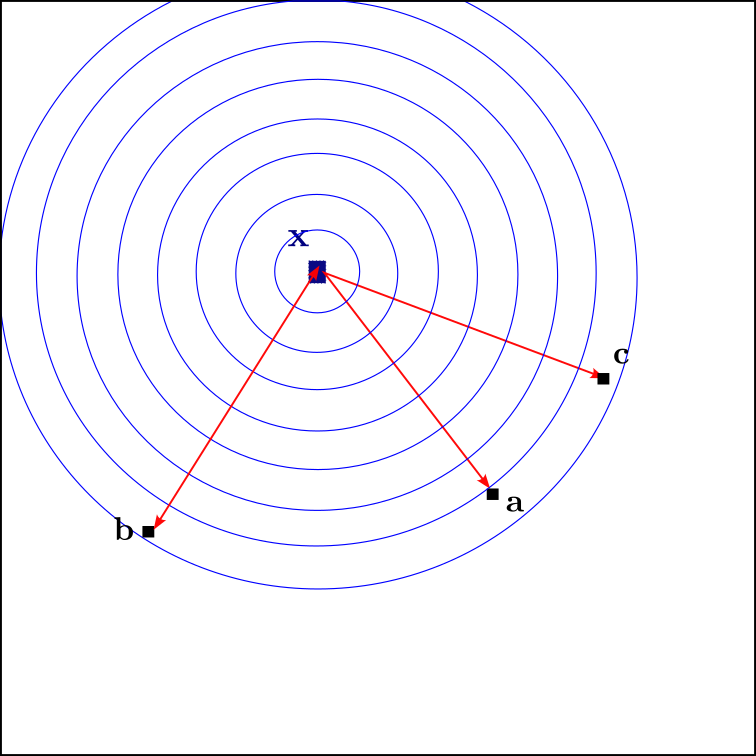}
\par\end{centering}
\caption{The robot at position $\mathbf{x}$ emits a sound received later by
three microphones $\mathbf{a}$, $\mathbf{b}$ and $\mathbf{c}$}
\label{fig:ellipseloc0}
\end{figure}

We have
\begin{equation}
\begin{array}{ccc}
\|\mathbf{x}-\mathbf{a}\| & = & c\cdot(t_{a}-t_{0})\\
\|\mathbf{x}-\mathbf{b}\| & = & c\cdot(t_{b}-t_{0})\\
\|\mathbf{x}-\mathbf{c}\| & = & c\cdot(t_{c}-t_{0})
\end{array}
\end{equation}
where $c$ is the sound speed and $t_{a},t_{b},t_{c}$ is the detection
time for microphones $\mathbf{a},\mathbf{b},\mathbf{c}$. We eliminate
$t_{0}$ which is unknown to get
\begin{equation}
\begin{array}{ccc}
\|\mathbf{x}-\mathbf{a}\|-\|\mathbf{x}-\mathbf{b}\| & = & c\cdot(t_{a}-t_{b})=\ell_{ab}\\
\|\mathbf{x}-\mathbf{a}\|-\|\mathbf{x}-\mathbf{c}\| & = & c\cdot(t_{a}-t_{c})=\ell_{ac}
\end{array}
\end{equation}
The quantity $\ell_{ab},\ell_{bc}$ are called \emph{pseudo-distances}.
We assume that we were able to measure the two pseudo distances to
get $\ell_{ab}\in[7.9,8.1]$ and $\ell_{ac}\in[3.9,4.1]$. The set
$\mathbb{X}$ of all feasible locations is defined by 
\begin{equation}
\begin{array}{ccccc}
\text{(i)} & \, & \|\mathbf{x}-\mathbf{a}\|-\|\mathbf{x}-\mathbf{b}\| & \in & [7.9,8.1]\\
\text{(ii)} &  & \|\mathbf{x}-\mathbf{a}\|-\|\mathbf{x}-\mathbf{c}\| & \in & [3.9,4.1]
\end{array}\label{eq:inequality:X}
\end{equation}

From Proposition \ref{prop:hyperbola:foci}, we get that $\mathbb{X}$
is defined by
\begin{equation}
\mathbb{X}=\mathbb{X}_{ab}\cap\mathbb{X}_{ac}
\end{equation}
where (see (\ref{eq:fabl})):
\begin{equation}
\mathbb{X}_{ab}:\left\{ \begin{array}{c}
\mathbf{f}_{\mathbf{a},\mathbf{b},8.1}(\mathbf{x})\leq0\\
\mathbf{f}_{\mathbf{a},\mathbf{b},7.9}(\mathbf{x})\geq0
\end{array}\right.
\end{equation}
and
\begin{equation}
\mathbb{X}_{ac}:\left\{ \begin{array}{c}
\mathbf{f}_{\mathbf{a},\mathbf{c},4.1}(\mathbf{x})\leq0\\
\mathbf{f}_{\mathbf{a},\mathbf{c},3.9}(\mathbf{x})\geq0
\end{array}\right.
\end{equation}
Using a paver, we get an inner and an outer approximations for the
set of $\mathbb{X}_{ab},\mathbb{X}_{ac}$ and $\mathbb{X}$. Figures
\ref{fig:tdoa1}, \ref{fig:tdoa3}, \ref{fig:tdoa5} have been generated
with a classical forward-backward contractor \citep{BenhamouICLP99}.
We observe a strong clustering effect with many uncertain boxes that
the separator is not able to classify.

\begin{figure}[H]
\begin{centering}
\includegraphics[width=8cm]{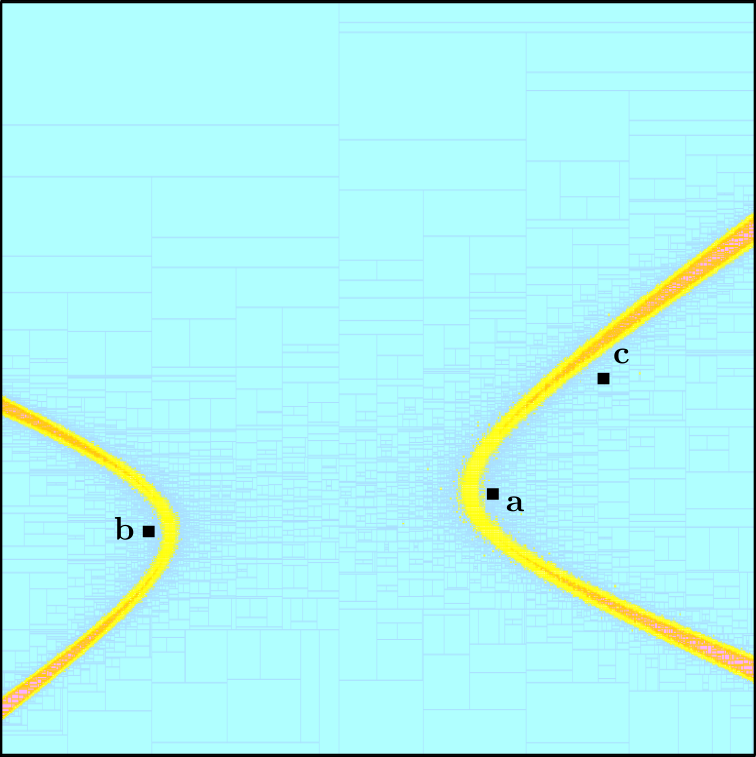}
\par\end{centering}
\caption{Set $\mathbb{X}_{ab}$ of positions consistent with microphones $\mathbf{a},\mathbf{b}$
(classic)}
 \label{fig:tdoa1}
\end{figure}

\begin{figure}[H]
\begin{centering}
\includegraphics[width=8cm]{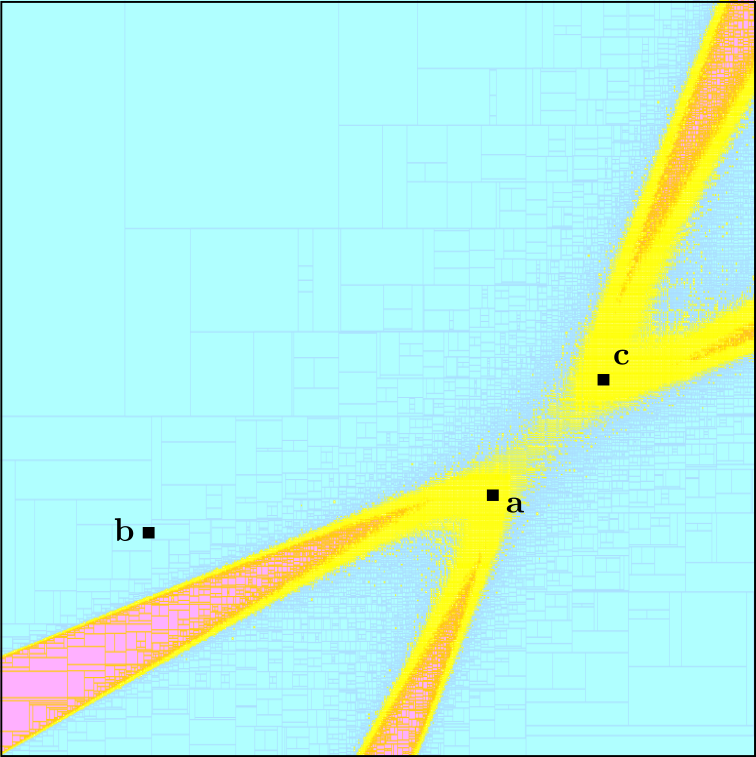}
\par\end{centering}
\caption{Set $\mathbb{X}_{ac}$ of positions consistent with $\mathbf{a},\mathbf{c}$
(classic)}
\label{fig:tdoa3}
\end{figure}

\begin{figure}[H]
\begin{centering}
\includegraphics[width=8cm]{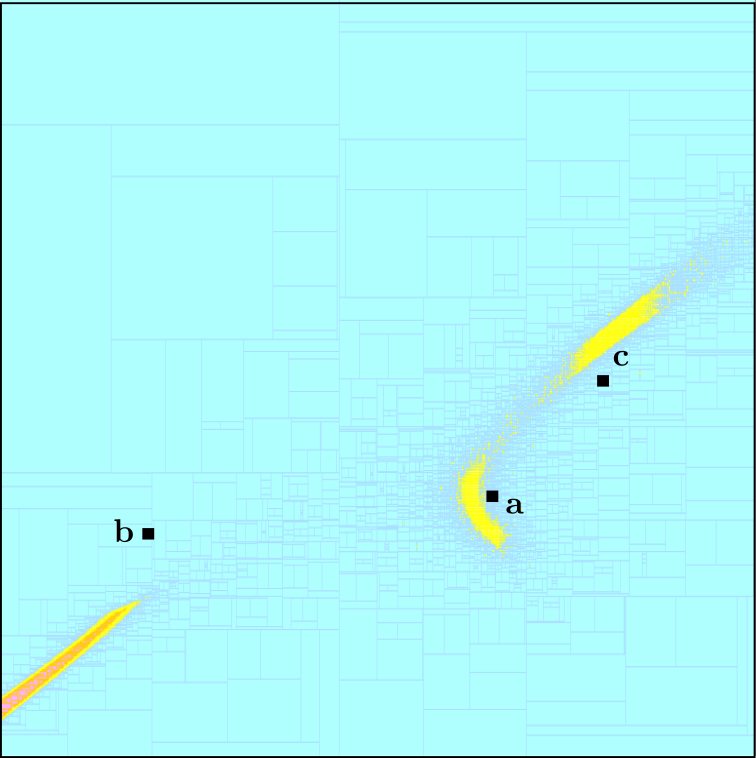}
\par\end{centering}
\caption{Set $\mathbb{X}$ of positions using the three microphones (classic)}
\label{fig:tdoa5}
\end{figure}

Figures \ref{fig:tdoa2}, \ref{fig:tdoa4}, \ref{fig:tdoa6} have
been generated using the minimal contractor (\ref{eq:ourminimal:ctc}).
For all figures, the frame box is $[0,20]\times[0,20]$ and the accuracy
is the same ($\varepsilon=0.05$). All results are guaranteed since
outward rounding is implemented \citep{revol2022testing}\citep{Revol17}.
The clustering effect almost disappeared. 

\begin{figure}[H]
\begin{centering}
\includegraphics[width=8cm]{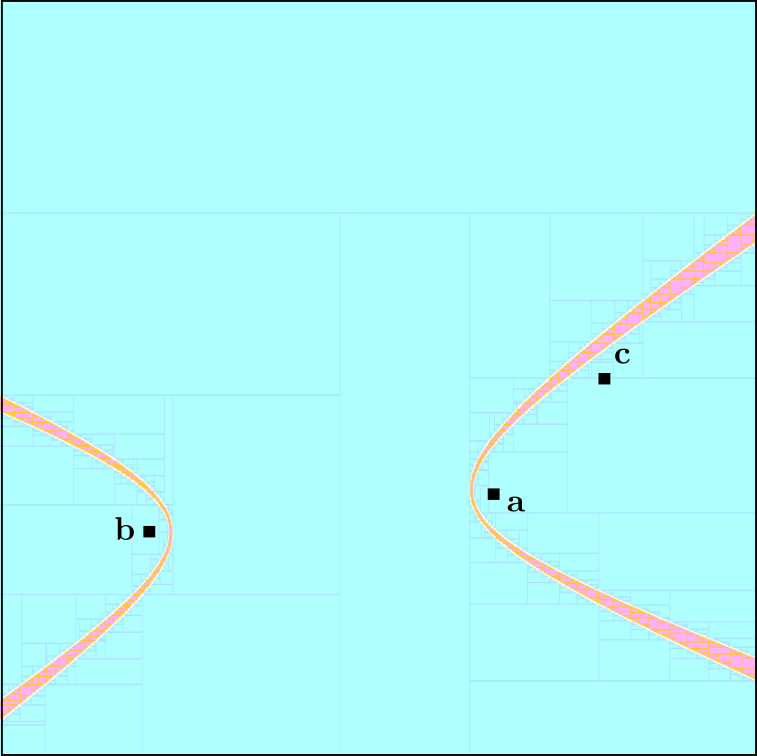}
\par\end{centering}
\caption{Set $\mathbb{X}_{ac}$ of positions consistent with microphones $\mathbf{a},\mathbf{c}$
(with the hyperbola separator)}
\label{fig:tdoa2}
\end{figure}

\begin{figure}[H]
\begin{centering}
\includegraphics[width=8cm]{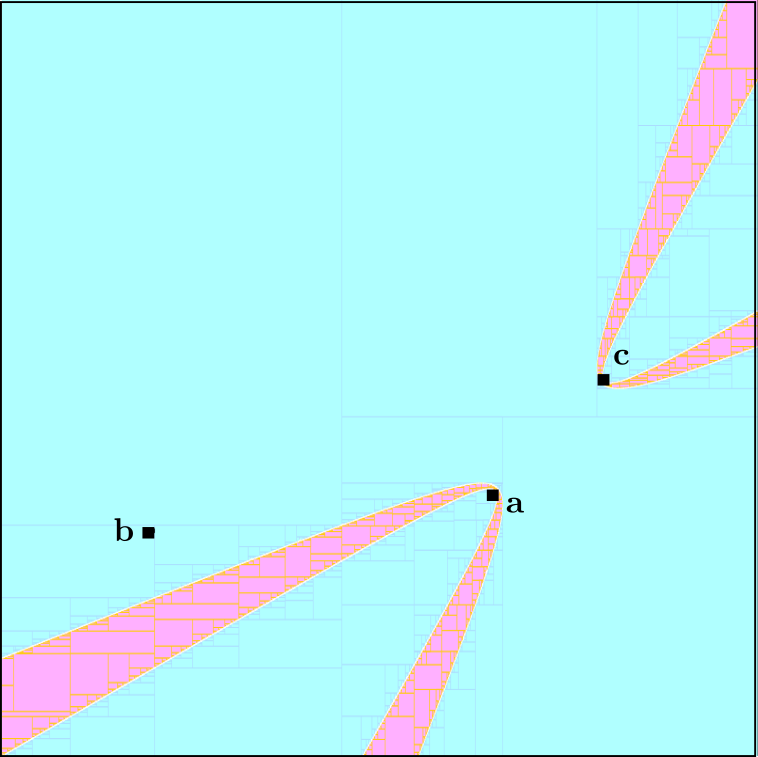}
\par\end{centering}
\caption{Set $\mathbb{X}_{ac}$ of positions consistent with microphones $\mathbf{a},\mathbf{c}$
(with the hyperbola separator)}
 \label{fig:tdoa4}
\end{figure}

\begin{figure}[H]
\begin{centering}
\includegraphics[width=8cm]{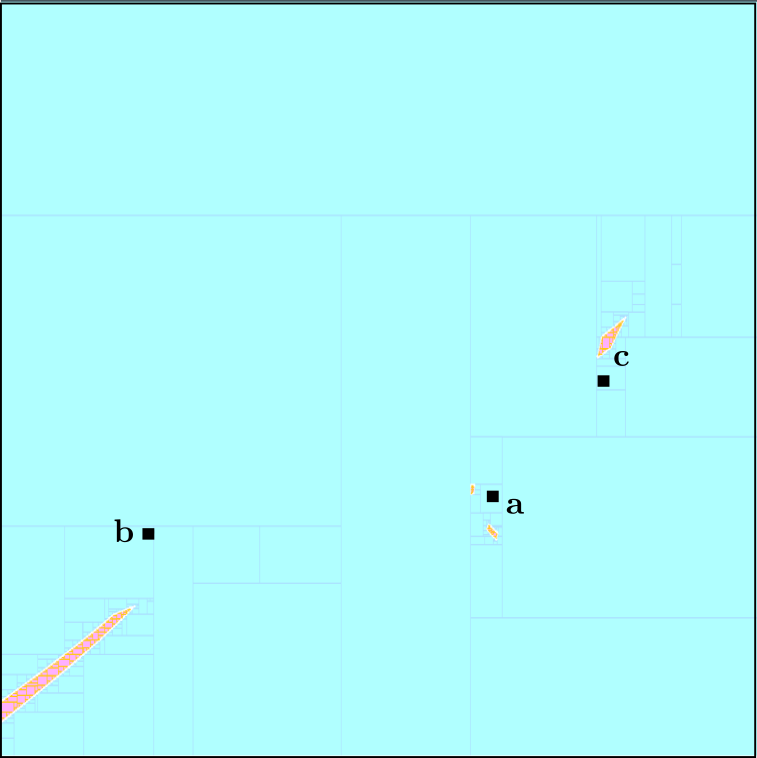}
\par\end{centering}
\caption{Set $\mathbb{X}$ of positions using the three microphones (with the
hyperbola separator)}
\label{fig:tdoa6}
\end{figure}

\section{Conclusion\label{sec:Conclusion}}

This paper has proposed a minimal contractor and a minimal separator
for an hyperbola area of the plane. The notion of actions derived
from hyperoctahedral symmetries allowed us to limit the analysis to
one portion of the constraint where the piece-wize monotonicity can
be assumed. The symmetries was used to extend the analysis to the
whole plane. 

The goal of this paper was also to provide a simple example which
illustrates the use of hyperoctahedral symmetries in order to build
minimal separators. Now, as shown in \citep{jaulin:quotient:2023},
the use of these symmetries is more interesting when we deal with
projection problems where quantifier elimination is needed. This type
of projection problem is indeed much more difficult to solve with
classical interval approaches \citep{ratschan}.

When we build an optimal contractor for a set $\mathbb{X}$ using
symmetries, the main difficulty is to find the portion of the set
that can be used to reconstruct $\mathbb{X}$ using the copy-paste
process allowed by the actions of the symmetries. For the hyperbola,
the pattern is a cardinal function and for the ellipse, it was a quarter
of the ellipse. But there is no general procedure to find the right
pattern. 

The Python code based on Codac \citep{codac} is given in \citep{jaulin:code:ctchyperbola:23}.

\bibliographystyle{plain}

\end{document}